\documentclass[11pt]{amsart}
\usepackage{amssymb,amsmath,amsthm}
\usepackage{pdfsync}
\usepackage{enumerate}
\usepackage{graphicx,color}
\newcommand{\N}{{\mathbb N}}
\newcommand{\R}{\mathbb{R}}

\newcommand{\Z}{{\mathbb{Z}}}

\newcommand{\Hilbert}{{\mathcal{H}}}  

\newtheorem{theorem}{Theorem}
\newtheorem{proposition}[theorem]{Proposition}
\newtheorem{definition}[theorem]{Definition}
\newtheorem{remark}[theorem]{Remark}
\newtheorem{question}[theorem]{Question}
\newtheorem{lemma}[theorem]{Lemma}
\newtheorem{corollary}[theorem]{Corollary}

\newcommand{\norm}[1]{\| #1\|}

\newcommand{\op}[1]{\!\!\mathop{\rm ~#1}\nolimits}

\newcommand{\Om}{\Omega}

\newcommand{\om}{\omega}

\renewcommand{\geq}{\geqslant}
\renewcommand{\leq}{\leqslant}

\newcommand{\abs}[1]{\left|#1\right|}
\newcommand{\OP}{{\rm Op}_\hbar}

\newcommand{\RM}{\mathbb{R}}
\newcommand{\h}{\hbar}
\newcommand{\f}{{\vec{f}}}

\usepackage{hyperref}

\begin{document}

\title[Spectral limits of semiclassical commuting operators]{Spectral
  limits of semiclassical commuting self\--adjoint operators}

\author{\'Alvaro Pelayo, \,\,\, San V\~u Ng\d oc}

\maketitle

\vspace{-0.5cm}
\begin{center}
    {\em Dedicated to Professor J.~M. Montesinos Amilibia, with
    admiration.}
\end{center}


\begin{abstract}

\noindent 
Using an abstract notion of semiclassical quantization for
self\--adjoint operators, we prove that the joint spectrum of a
collection of commuting semiclassical self\--adjoint operators
converges to the classical spectrum given by the joint image of the
principal symbols, in the semiclassical limit.  This includes
Berezin\--Toeplitz quantization and certain cases of
$\hbar$\--pseudodifferential quantization, for instance when the
symbols are uniformly bounded, and extends a result by L.~Polterovich
and the authors.  In the last part of the paper we review the recent
solution to the inverse problem for quantum integrable systems with
periodic Hamiltonians, and explain how it also follows from the main
result in this paper.
\end{abstract}


\section{Introduction} \label{intro} In inverse spectral problems one
tries to recover geometric (or ``classical'') information from the
spectrum of a ``quantum'' operator. For instance, does the spectrum of
the Laplacian on a bounded euclidean domain completely determine the
geometry of the domain ? The problem goes back to S. Bochner and
H. Weyl \cite{We1911,We1912} in the late nineteenth and early
twentieth century, and was made popular, in the context of Riemannian
geometry, in M. Kac's famous article on ``can you hear the shape of a
drum", \cite{Ka66}, who attributes the origin of the question to
Bochner.

In this paper we will deal with a quite general setting of
semiclassical self-adjoint operators. Roughly speaking, a quantum
operator will be a family of operators $(T_\hbar)$ depending on a
small real parameter $\hbar>0$ reminiscent of the Planck constant. To
each such operator, one defines its ``classical limit'' to be a smooth
function on a smooth manifold (the phase space), called the
\emph{principal symbol} of the operator. The semiclassical inverse
problem is then the following.

\begin{question}{\bf (Semiclassical Inverse Spectral
    Problem)} \label{q1} Given the semiclassical joint
  spectrum $$(X_{\hbar})_{\hbar>0}\subset\R^d$$ of a quantum system of
  commuting semiclassical operators
  $$T_1:=(T_{1,\hbar})_{\hbar>0},\ldots,T_d:=(T_{d,\hbar})_{\hbar>0},$$
  how much can one recover about the classical system given by the
  principal symbols $f_1,\ldots,f_d$ of $T_1,\ldots,T_d$ ?
\end{question}

Of course, a complete answer to this question would be to fully
recover the principal symbols $f_1,\ldots,f_d$ themselves.

One can hope to obtain such general results by combining the use of
microlocal and symplectic techniques in the spirit of Duistermaat,
Helffer, H\"ormander, Sj\"ostrand, etc. However, to date only a
handful of results are known in this direction, see~\cite{CdV, CdV2,
  VN11, ChPeVN13, LF14, LFPeVN15}. If one restricts the class of
operators to be Schr\"odinger operators, whose principal symbol is of
the form $\xi^2+V(x)$, then the question amounts to recovering the
potential $V$; this has attracted a lot of mathematicians, and is
still an active area of research, see~\cite{He09, CdVGu11, GuPa09}.

In this paper we prove a general result giving a partial answer to
Question~\ref{q1}, inspired by previous works of Colin de
Verdiere~\cite{CdV, CdV2}, Polterovich, and the
authors~\cite{PePoVN14}, which says that even though we do not know
how to recover the principal symbols themselves, we can recover the
closure of their joint image, which is a subset of the affine space
$\mathbb{R}^d$. This gives a rigorous proof of the quantum mechanical
principle that says that: ``in the high frequency limit $\hbar\to 0$,
the spectrum of a quantum system converges to the numerical range of
its associated classical system".

\begin{theorem}
  \label{theo:main0}
  Let $I \subset (0,\,1]$ be a set with a limit point at $0$.  Then
  the limit set of the joint spectrum of a family of pairwise
  commuting self\--adjoint semiclassical operators
  \[
  T_1:=(T_{1,\h})_{\h\in I},\,\ldots,\, T_d:=(T_{d,\h})_{\h \in I}
  \]
  is the classical spectrum $\mathcal{S} \subset \mathbb{R}^d$ of
  $T_1,\dots, T_d$, that is, the closure of the joint image of the
  principal symbols of $T_1,\ldots,T_d$.
\end{theorem}

An illustration of the convergence statement in
Theorem~\ref{theo:main0} is depicted in Figure~\ref{newfigure2}, which
shows the joint spectrum of the ``normalized'' Quantum Spherical
Pendulum\footnote{Instead of the standard energy and momentum
  operators, the energy is replaced by its square root, in order to
  obtain symbols which are both (asymptotically) homogeneous of degree
  one.}

\begin{figure}[h]
  \centering \label{process2}
  \includegraphics[width=1\textwidth]{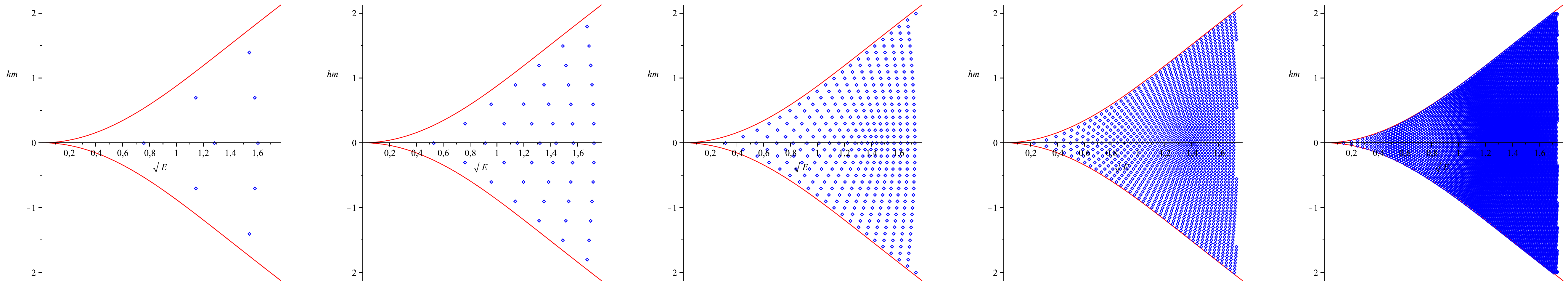}
  \caption{The dots in the figures form the semiclassical joint
    spectrum of the ``normalized'' Quantum Spherical Pendulum for the
    values of the Planck constant: $\hbar=0.7, 0.5, 0.3, 0.05, 0.02$.
    As $\hbar \to 0$, the semiclassical joint spectrum fills the
    inside of the red curve, which is the boundary of the classical
    spectrum of the system; this gives an illustration of the
    convergence stated in Theorem~\ref{theo:main0}.}
  \label{newfigure2}
\end{figure}

In~\cite{PePoVN14} an analogous statement was proved, but taking the
convex full on both the quantum and the classical spectrum.  We
achieve this improvement by introducing a new hypothesis, which takes
the form of the following seemingly simple axiom for the abstract
semiclassical quantization: for any symbol $f$, one should have
\[
\norm{\OP(f)^2\ - \OP(f^2)} = \mathcal{O}(\hbar).
\]
where $\OP$ denotes the quantization operation.  We refer to
Theorem~\ref{theo:main} for a detailed version of the above statement,
to Definition~\ref{BT} for the abstract notion of semiclassical
operators we use, and the upcoming sections for the necessary
preliminaries. The abstract notion we use, and hence the theorem,
apply to Berezin\--Toeplitz operators on compact manifolds, and
certain classes of pseuodifferential operators (this is explained in
Remark~\ref{importantremark}), for instance those with uniformly
bounded derivatives.

As we will explain, Theorem~\ref{theo:main} implies, in combination
with a theorem of Atiyah\--Guillemin\--Sternberg and Delzant, a
solution to the inverse problem for quantum toric integrable systems,
which recovers a recent result of Charles and the
authors~\cite{ChPeVN13}. This result was proved again shortly after by
Polterovich and the authors~\cite{PePoVN14} with a different method,
which is in fact the one which serves as inspiration for
Theorem~\ref{theo:main}, in combination with ideas introduced by Le
Floch and the authors in~\cite{LFPeVN15}.

\section{Semiclassical operators and an abstract semiclassical
  quantization}

We review Berezin\--Toeplitz quantization,
$\hbar$\--pseudodifferential quantization, and then introduce an
abstract notion of semiclassical quantization which includes the
former, and certain classes of the latter. This abstract notion is
inspired by, and extends, a notion introduced by Polterovich and the
authors~\cite{PePoVN14} and as we will see in Section~\ref{ip} it
allows us to prove a stronger convergence result in certain cases.

\subsection{Berezin\--Toeplitz operators} \label{bz}

The microlocal analysis of Toeplitz operators is rapidly evolving
nowadays, see for instance~\cite{BorPauUri,Ch2003b, Ch2006b, Ch2006a,
  Ch2007,MaMar2008,Schli2010} following the pioneer work of Boutet de
Movel and Guillemin~\cite{BG81}.

Let us recall the basic facts we need on connections of Hermitian line
bundles (a good reference for this material are Duistermaat's notes
\cite{Du2004}).  With the help of these facts we will introduce a
fundamental notion in both geometry and analysis, that of a prequantum
line bundle.

Let $M$ be a smooth manifold. Let ${\mathcal{L}} \rightarrow M$ be a
Hermitian line bundle over $M$. That is, ${\mathcal{L}} \rightarrow M$
is a complex line bundle over $M$ which is endowed with a Hermitian
metric. Denote by ${\rm C} ^{ \infty} (M, {\mathcal{L}})$ the space of
smooth sections of this bundle, and by $\Om^{1} (M, {\mathcal{L}})$
the space of smooth ${\mathcal{L}}$\--valued $1$-forms.  A
\emph{connection of ${\mathcal{L}}$} is a linear operator
$\nabla : {\rm C}^{\infty} (M, {\mathcal{L}}) \rightarrow \Om^{1} (M,
{\mathcal{L}})$ which satisfies Leibniz's rule, that is,
$$ \nabla (f s ) = {\rm d}f \otimes s + f \nabla s$$
for all smooth functions $f \in {\rm C}^{\infty} (M)$ and all smooth
sections $s \in \op{C}^{\infty}(M, {\mathcal{L}}) $.

Let $X$ be a smooth vector field on $M$.  The \emph{covariant
  derivative of the section $s$} of ${\mathcal{L}}$ with respect to
the vector field $X$ is given by the formula
$ \nabla_X s = \nabla s(X).  $ The smooth $2$\--form $R$ of $M$
defined by the equation
$$ R (X, Y) = [\nabla_X , \nabla_Y ] - \nabla_{[X, Y]} $$
for any vector fields $X$, $ Y$ of $M$ is called the \emph{curvature
  of the connection}. Let $(\cdot,\cdot)$ denote the Hermitian scalar
product.  We say that the connection is \emph{compatible with the
  Hermitian structure} if
$ {\rm d}(s, t) = ( \nabla s , t) + ( s, \nabla t),$ for any smooth
sections $s$ and $t$ of the line bundle ${\mathcal{L}}$.  In this
case, $R = \frac{1}{{\rm i}} \om$ where $\om$ is
real\--valued. Throughout the present paper all connections considered
are implicitly assumed to be compatible with the metric.

Assume that $\frac{1}{{\rm i}} \om$ is the curvature of a Hermitian
line bundle connection. Then the cohomology class of the form
$\om/ 2 \pi $ is integral, that it, it lies in the image of the
canonical homomorphism $$H^2 (M, \Z) \rightarrow H^2 (M, \R).$$
Conversely, for any smooth $2$\--form $\om \in \Om^2 (M, \R)$ for
which $[\om]/2\pi$ is integral there is a Hermitian line bundle
${\mathcal{L}} \to M$ endowed with a connection $\nabla$ whose
curvature is $\frac{1}{{\rm i}} \om$. Moreover, the line bundle
${\mathcal{L}}$ and $\nabla$ are unique up to isomorphisms. For a
proof of these results, we refer the reader to \cite[Theorem
10.1]{Du2004} or \cite[Section 15.3]{Duistermaat_index}. Now we are
ready to recall the following essential definition.

\begin{definition}
  Let $(M, \om)$ be a symplectic manifold.
  \begin{itemize}
  \item A \emph{prequantum bundle} on $M$ is a Hermitian line bundle
    ${\mathcal{L}} \rightarrow M$ with a connection of curvature
    $\frac{1}{{\rm i}} \om$.

    In this case we say that the symplectic manifold $(M,\omega)$ is
    \emph{prequantizable}.
  \item A \emph{prequantum bundle automorphism} is a vector bundle
    automorphism of ${\mathcal{L}} \to M$ which preserves both the
    metric and the connection.
  \end{itemize}
\end{definition}

Let us know consider a prequantum line bundle
${\mathcal{L}} \rightarrow M$ over the symplectic manifold $M$.  Let
$G$ be a Lie group with Lie algebra $\mathfrak{g}$. Assume that $G$
acts on ${\mathcal{L}}$ by prequantum bundle automorphisms, as defined
above. This $G$\--action lifts a $G$\--action on $M$. For
$X \in \mathfrak{g}$ let $X^\sharp$ denote the infinitesimal action of
$X$ on $M$.  The latter $G$\--action is Hamiltonian with momentum map
$F$ given by the following condition: the action induced by
$\mathfrak g$ on ${\rm C}^{\infty}(M, \mathcal{L})$ is given by the
Kostant\--Souriau operators
\begin{equation} \label{eq:KS} f \mapsto \nabla_{X^\sharp}f + {\rm i}
  \langle F , X \rangle f , \qquad X \in \mathfrak{g}.
\end{equation}
and $\nabla$ denotes the covariant derivative of the prequantum bundle
(cf. \cite[Proposition 15.2]{Duistermaat_index}).  If both the Lie
group $G$ and the manifold $M$ are connected, then the $G$\--action on
${\mathcal{L}}$ is conversely determined by the action on $M$ and by
the momentum map $F$. It is important to notice that we cannot obtain
every momentum map generating a given action in this manner (such
momentum maps correspond to the Lie algebra representations on the
prequantum bundle by means of (\ref{eq:KS}).

Suppose that $(M,\,\omega)$ is a prequantizable closed (that is both
compact and with no boundary) symplectic manifold. Let
$\mathcal{L} \to M$ be a prequantum line bundle, and assume that $M$
admits a complex structure that is compatible with the symplectic form
$\om$. In fact, $(M,\omega)$ is a K\"ahler manifold.  The holomorphic
structure of $\mathcal{L} \to M$ is uniquely determined by the
compatibility condition with the connection.

Consider a positive integer $k=1/\hbar$ and let $\mathcal{L}^k$ denote
the $k$th tensor power of the line bundle $\mathcal{L}$.  We write
 $$\mathcal{H}_{\hbar}:={H}^0(M,\mathcal{L}^k)
 $$ 
 for the space of holomorphic sections of $\mathcal{L}^k$.  The space
 $\mathcal{H}_{\hbar}$ is a finite dimensional subspace of the Hilbert
 space ${L}^2(M, \mathcal{L}^k)$ (this is because $M$ is a closed
 manifold).  If $\lambda$ denotes the Liouville measure of $M$, the
 scalar product is given by integration of the Hermitian (pointwise)
 scalar product of sections against $\lambda$.

 \begin{definition} \label{tp} Let $\Pi_{\hbar}$ be the surjective
   orthogonal projector
   $$ {\rm L}^2(M, \mathcal{L}^k) \to {\mathcal{H}}_{\hbar}.
 $$  
 \begin{itemize}
 \item A semiclassical \emph{Berezin\--Toeplitz operator} is any
   sequence of the form
$$T:=(T_{\hbar}:= \Pi_{\hbar} f(\cdot,\, k) \colon \Hilbert_{\hbar} \rightarrow \Hilbert_{\hbar})_{\hbar=1/k, \; k \in \mathbb{N}^*}$$ 
where the multiplication operator $f(\cdot,\,k)$ is a sequence in
${\rm C}^{\infty}(M)$ with an asymptotic expansion
 $$
 f_0 + k^{-1} f_1 + k^{-2}f_2 + \cdots
 $$
 for the ${\rm C}^{\infty}$ topology.
\item The first coefficient $f_0$ is called the {\em principal symbol}
  of $T$.
\end{itemize}
\end{definition}

\subsection{Pseudodifferential operators} \label{az}

\label{sec:pseudodifferential}

$\hbar$\--pseudodifferential operators acting on the Hilbert space
$\mathcal{H}=L^2(\R^n)$ give a semiclassical quantization of the
manifold $M=\R^{2n}$; this is a semiclassical version of the one given
by homogeneous pseudodifferential operators, see for
instance~\cite{DiSj99} or~\cite{Zw2012}.

Let $\mathcal{A}_0$ be the H\"ormander class whose elements are the
functions $f$ in the space $\textup{C}^\infty(\R^{2n}_{(x,\xi)})$ such
that the following holds: there is $m\in\R$ for which
\begin{equation}
  |\partial_{(x,\xi)}^\alpha f|\leq C_\alpha\langle (x,\xi)\rangle^m
  \label{equ:hormander}
\end{equation}
for every $\alpha \in \mathbb{N}^{2n}$ (here the notation
$\langle z \rangle$ stands for $(1+|z|^2)^{1/2}$). Symbolic calculus
for pseudodifferential operators is known to hold in the case when the
symbols of the operators are in $\mathcal{A}_0$ (for instance).

\begin{definition} \label{wq} Let $f\in\mathcal{A}_0$.  The \emph{Weyl
    quantization of} $f$ is given on the Schwartz space
  $\mathcal{S}(\R^n)$ by the expression:
  \begin{equation}
    \label{equ:weyl}
    (\OP(f) u)(x) := \frac{1}{(2\pi\hbar)^{n}}\int_{\R^{n}} \int_{\R^{n}}
    \textup{e}^{\frac{{\rm i}}{\hbar}((x-y)\cdot\xi)} f({\textstyle\frac{x+y}{2}},\xi)u(y)\,\, {\rm d}y \,{\rm d}\xi.
    \nonumber
  \end{equation}
\end{definition}

This definition is commonly used to define $\h$-pseudodifferential
operators on $\R^n$. It can also give a semiclassical quantization of
a cotangent bundle $M=T^*X$, where $X$ is a smooth $n$\--dimensional
closed manifold with a smooth density, as follows. Let $X$ be covered
by a collection of smooth charts $$\Big\{U_1,\dots,U_N\Big\},$$ where
each $U_i$ is over a convex bounded domain of the Euclidean space
$\R^n$ (equipped with the Lebesgue measure). By standard manifold
theory, there exists a partition of unity
$$\chi_1^2,\dots,\chi_N^2$$ which is subordinated to the cover $\{U_1,\dots,U_N\}.$ In this case, $\mathcal{A}_0$ is the space of
functions $f\in\textup{C}^\infty(\textup{T}^*X)$ satisfying, for all
$(x,\xi)\in \textup{T}^*X,\,\alpha\in \mathbb{N}^n$, and for some
$m\in\R$, the condition:
\begin{equation}
  \abs{\partial_x^\alpha\partial_\xi^\beta f(x,\xi)} \leq C_\alpha
  \langle \xi \rangle ^{m-|\beta|}
  \label{equ:nirenberg}
\end{equation}

Recall Definition~\ref{wq} and let $\OP^j(f)$ be the Weyl quantization
in $U_j$. Define:
\begin{equation} \label{eq-weyl-manifold} \OP(f) u:= \sum_{j=1}^N
  \chi_j \cdot \OP^j(f) (\chi_j u)\;, \,\,\,\,\,
  u\in\textup{C}^\infty(X)\,, \nonumber
\end{equation}
which is a pseudodifferential operator on $X$. The principal symbol of
this operator is the smooth function
$$
f:=\sum_{i=1}^N f\chi_j^2.
$$

\begin{definition} \label{ps} Let $X$ be either $\R^n$, or a closed
  manifold, as above. Let $(f_\hbar)_{\hbar\in(0,1]}$ be a family of
  elements of $\mathcal{A}_0$ such that the
  estimate~\eqref{equ:hormander} (in the case $X=\R^{n}$)
  or~\eqref{equ:nirenberg} (if $X$ is a closed manifold) holds
  uniformly for $\hbar\in(0,1]$. Then the family
  $$T:=(\OP(f_\hbar))_{\hbar \in (0,\,1]}$$ is called a
  \emph{semiclassical} $\hbar$\--\emph{pseudodifferential operator on}
  $X$.
\end{definition}

The above definition may also be made for a subset $I \subset (0,\,1]$
which has a limit point at $0$, for instance
$$I=\Big\{\frac{1}{k} \,\,|\,\, k \in \mathbb{N}^*\Big\}.$$

\subsection{Abstract semiclassical quantization}

The results presented in this paper hold for both pseudodifferential
and Berezin-Toeplitz quantization. In fact, they only require a few
key properties, and it is interesting to state them in an abstract
way, as follows.

Let $I \subset (0,1]$ be a set that accumulates at $0$.  Suppose that
$M$ is a connected manifold (closed or open) and let $\mathcal{A}_0$
be a subalgebra of the algebra of smooth functions
$\textup{C}^\infty(M;\R)$ containing all constants as well as all
compactly supported functions.

For a complex Hilbert space $\mathcal{H}$ we denote by
$\mathcal{L}(\mathcal{H})$ the set of all linear self\--adjoint
operators on $\mathcal{H}$ (bounded or unbounded). The following
definition is essentially the same as in \cite{PePoVN14} with the
exception of the new Axiom {(Q\ref{item:square})}, which is needed for
the proof of our main result.

\begin{definition} \label{BT}

  A \emph{semiclassical quantization} of the pair $(M,\mathcal{A}_0)$
  is given by:
  \begin{itemize}
  \item a family of complex Hilbert spaces
    $\mathcal{H}_{\hbar},\; \hbar\in I$, and
  \item a family of $\R$\--linear maps
    $\OP \colon \mathcal{A}_0 \to \mathcal{L}(\mathcal{H}_{\hbar})$,
  \end{itemize}
  that satisfy the following axioms, where $f,g \in \mathcal{A}_0$:
  \begin{enumerate}[{\rm (Q1)}]
  \item \label{item:one}
    $\norm{\OP (1) - {\rm Id}} = \mathcal{O}(\hbar)$
    \emph{(normalization)};
  \item \label{item:garding} for every function $f \geq 0$ there is a
    constant $C_f$ for which $\OP (f) \geq -C_f \hbar $
    \emph{(quasi-positivity)};
  \item \label{item:norm} if $f\in\mathcal{A}_0$ is that such that
    $f\neq 0$ and also has compact support, then
  $$\liminf_{\hbar\to 0}
  \norm{\OP(f)}>0$$ \emph{(non-degeneracy)};
\item \label{item:symbolic} if $g$ has compact support, then the
  operator $\OP(f) \circ \OP(g)$ is bounded for every $f$, and
  \[ \norm{\OP(f) \circ \OP(g) - \OP(fg)} = \mathcal{O}(\hbar),
  \] \emph{(product formula)};
\item \label{item:square} if $f\in \mathcal{A}_0$, then
  \[
  \norm{\OP(f)^2\ - \OP(f^2)} = \mathcal{O}(\hbar),
  \]
  \emph{(square formula).}
\end{enumerate}
We say that a manifold is \emph{quantizable} if it has a semiclassical
quantization.
\end{definition}

Let $\mathcal{A}_I$ be the algebra whose elements are collections
$\vec{f} =(f_\hbar)_{\hbar \in I}$, $f_\hbar \in \mathcal{A}_0$, that
satisfy that for each $\f$ there is $f_0 \in \mathcal{A}_0$ such that
\begin{equation}\label{eq-vector}
  f_\hbar = f_0+ \hbar f_{1,\hbar}\;,
\end{equation}
where $f_{1,\hbar}$ is uniformly bounded in the parameter $\hbar$ as
well as supported in the same compact subset $K(\f)$ of $M$.

\begin{definition} \label{xd} A \emph{semiclassical operator} is an
  element in the image of the map
  $$\text{Op}: \mathcal{A}_I \to \prod_{\hbar \in I}
  \mathcal{L}(\mathcal{H}_\hbar)$$ defined by
  $$
  \f=(f_\hbar) \mapsto (\OP(f_\hbar))\;.$$
\end{definition}
  
  \begin{definition}\label{xd2}
    Let $\f \in \mathcal{A}_I$. The function $f_0 \in \mathcal{A}_0$
    given by \eqref{eq-vector} is called the \emph{principal symbol}
    of $\text{Op}(\f)$.
  \end{definition}

  It follows from the axioms in Definition~\ref{xd} that the principal
  symbol in Definition~\ref{xd2} is uniquely defined,
  see~\cite{PePoVN14}.

\label{sec:semicl-oper}
Notice that in the definition of semiclassical operators, the manifold
$M$ is not required to be symplectic.  The following proposition gives
the two major examples of semiclassical operators, for which the phase
space $M$ is symplectic.
\begin{proposition}
  \label{prop:major-examples}
  \begin{enumerate}
  \item Semiclassical Berezin-Toeplitz operators satisfy the axioms
    \emph{(Q1--Q5)}.
  \item Semiclassical pseudodifferential operators which mildly
    depends on $\hbar$ satisfy the axioms \emph{(Q1--Q4)}. Here we say
    that $f_\hbar\in \mathcal{A}_I$ mildly depends on $\hbar$ if
    $f_\h$ can be written
    \[
    f_\hbar(x,\xi) = f_0(x,\xi) + \hbar f_{1,\hbar}(x,\xi),
    \]
    where all $f_{1,\hbar}(x,\xi)$ are both uniformly bounded in
    $\hbar$ as well as compactly supported in the same set.
  \item Semiclassical pseudodifferential operators which are uniformly
    bounded satisfy the axiom \emph{(Q\ref{item:square})}. More
    precisely, by uniformly bounded we mean that for any
    $f_\hbar\in \mathcal{A}_I$, and every
    $\alpha \in \mathbb{N}^{2n}$, there is a constant $C_\alpha$ such
    that
    \[
    |\partial_{(x,\xi)}^\alpha f_\hbar(x,\xi)|\leq C_\alpha \qquad
    \forall \hbar\in(0,1].
    \]
  \end{enumerate}
\end{proposition}
A proof of this can be found in most introductory papers or books on
the subject. For instance, for Berezin-Toeplitz operators, one can
refer to \cite{BorPauUri,Ch2003b, Ch2006b, Ch2006a,
  Ch2007,MaMar2008,Schli2010}, and for pseudodifferential operators to
the books~\cite{DiSj99} or~\cite{Zw2012}. Here we do not claim to have
optimal hypothesis. For instance, the assumption on mild dependence on
$\h$ can certainly be weakened.

\begin{remark} \label{importantremark} For pseudodifferential
  operators, axiom \emph{(Q\ref{item:square})} (square formula) is
  more restrictive than the others; it does not hold for all classes
  of symbols. In order to use the results for differential operators
  like the Laplacian, one would need first a microlocalization
  estimate in order to truncate the original operators and hence
  transform them into uniformly bounded pseudodifferential
  operators. Such a truncation procedure is common in microlocal
  analysis (see for instance~\cite[Chapter 10]{DiSj99}).
\end{remark}

The following lemma is a consequence of the axioms (Q1--Q4):
\begin{lemma}[\mbox{\cite[Lemma 11]{PePoVN14}}] \label{lem:max} Take
  any $\f= (f_\hbar) \in\mathcal{A}_I$ with principal part $f_0$, and
  let $(\OP(f_\hbar))$ be the corresponding semiclassical operator.
  Let $\lambda_{\inf}(\hbar)\in [-\infty,+\infty)$ denote the infimum
  of the spectrum of $\OP(f_\hbar)$.  Then
  \begin{eqnarray} \label{eq:0} \lim_{\hbar \to 0}
    \lambda_{\inf}(\hbar)= \inf_M f_0.
  \end{eqnarray}
\end{lemma}

\section{Joint Spectrum of a family of semiclassical operators}

We recall that to any self-adjoint operator $A$ on a Hilbert space,
the spectral theorem associates a projector-valued measure $\mu_A$,
called the spectral measure, such that \[ A = \int_{\R} t d(\mu_A)(t),
\]
and whose support is the spectrum of $A$. A similar theory holds for
commuting operators.  The self\--adjoint operators $S_1,\ldots,S_d$
are said to be mutually commuting if their corresponding spectral
measures $\mu_1,\dots,\mu_d$ pairwise commute. In this case we may
then define the joint spectral measure
 $$\mu := \mu_1\otimes\cdots\otimes\mu_d
 $$
 on $\R^d$. The \emph{joint spectrum} of $(S_1,\ldots,S_d)$ is the
 support of the joint spectral measure, that is:
 \begin{multline*}
   c\in \op{JointSpec}(S_1,\, \ldots, \, S_d) \iff \\
   \forall \epsilon>0, \quad \mu_1([c_1-\epsilon, c_1+\epsilon]) \circ
   \cdots \circ \mu_d([c_d-\epsilon, c_d+\epsilon]) \neq 0.
 \end{multline*}

 In this paper we are interested in the joint spectrum of
 semiclassical operators, which is defined as follows. For
 $j \in \{1,\ldots,d\}$ let $$T_j=(T_{j,\hbar})_{\hbar \in I}$$ be
 semiclassical operators (as in Definitions~\ref{tp} or \ref{ps}) on
 Hilbert spaces $(\mathcal{H}_\h)_{\hbar \in I}.$ We assume that for
 any fixed $h\in I$, the self-adjoint operators
 $T_{1,\hbar}, \dots, T_{d,\hbar}$ are mutually commuting.  For fixed
 $\hbar$, the \emph{joint spectrum} of
 $(T_{1,\hbar},\ldots,T_{d,\hbar})$ is as before the support of the
 joint spectral measure. For instance, if $\mathcal{H}_\h$ is finite
 dimensional (eg. in the case of Berezin-Toeplitz quantization on a
 closed K\"ahler manifold), then
$$
\op{JointSpec}(T_{1,\hbar},\ldots,T_{d,\hbar})
$$ 
is the set
$$
\Big\{(\lambda_1,\ldots,\lambda_d)\in \R^d\,\, |\,\, \exists v\neq
0,\,\, T_{j,\hbar} v = \lambda_j v, \forall j=1,\ldots,d \Big\}.
 $$
 We define the \emph{joint spectrum} of the semiclassical operators
 $(T_1,\ldots,T_d)$ to be the collection of all joint spectra of
 $(T_{1,\hbar},\ldots,T_{d,\hbar})$, $\hbar \in I$.

 \section{The inverse problem for commuting operators} \label{ip}

 \subsection{Convergence to classical spectrum}

 Following the physicists, we use the following definition.
 \begin{definition}
   We call \emph{classical spectrum} of $(T_1,\dots, T_d)$ the closure
   of the image $F(M)\subset \R^d,$ where $F=(f_1,\dots,f_d)$ is the
   map of principal symbols of $T_1,\dots,T_d$.
 \end{definition}

 In order to state the convergence results for the semiclassical
 spectrum of a collection of operators, we need to use a notion of
 limit for subsets of $\R^n$.

\begin{definition}
  Let $(A_\h)_{\hbar \in I}$ be a family of subsets of $\mathbb{R}^n$,
  where $I\subset(0,1]$ is a set which accumulates at $0$. The
  \emph{limit set of $(A_\h)_{\hbar \in I}$} is the subset
  $A_0\subset \R^n$ defined by
  \[
  a\in A_0 \iff \forall\epsilon>0, \forall \h_0\in I, \exists \h\in I,
  \h\leq \h_0, \text{ such that } A_\h \cap B(a,\epsilon) \neq
  \emptyset.
  \]
  Here $B(a,\epsilon)$ is the euclidean ball around $a$ of radius
  $\epsilon$.
\end{definition}

In the case of uniformly bounded subsets of $\R^n$, the limit set is
in fact a limit in the sense of the Hausdorff distance, which we
recall now.  Let $\|\cdot\|$ be the euclidean norm in
$\mathbb{R}^n$. For any $\epsilon>0$ and any subset $X$ of $\R^n$, we
denote by $X_{\epsilon}$ the set
$$\bigcup_{x \in X} \Big\{m \in \R^n\, \, | \,\, \|x - m \|
\leq \epsilon\Big\}.$$
The \emph{Hausdorff distance} between two subsets $A$ and $B$ of
$\R^n$ is the number
 $$
 \inf\Big\{\epsilon > 0\,\, | \,\ A \subseteq B_\epsilon \ \mbox{and}\
 B \subseteq A_\epsilon\Big\}.
$$
We denote it by ${\rm d}_H(A,\,B)$.

\begin{definition}
  Let $(A_{\hbar})_{\hbar \in I}$ and $(B_{\hbar})_{\hbar \in I}$ be
  families of uniformly bounded subsets of $\mathbb{R}^n$, where
  $I\subset(0,1]$ is a set which accumulates at $0$.
  \begin{itemize}
  \item Fix $N\in\N$. We say that
    $$A_{\hbar} = B_{\hbar} + \mathcal{O}(\hbar^{N})$$ if there exists
    a constant $C>0$ such that
    $ {\rm d}_H(A_{\hbar},\,B_{\hbar})\leq C\hbar^{N}$ for every
    $\hbar \in I$.
  \item We say that $$A_\h = B_\h + \mathcal{O}(\h^{\infty})$$ if
    ${\rm d}_H(A_\h,\,B_\h)=\mathcal{O}(\h^{N})$ for every
    $N \in \mathbb{N}^*$.
  \item Let $A_0\subset\R^n$. We say that $A_0$ is a Hausdorff limit
    of $(A_{\hbar})_{\hbar \in I}$ if
    \[
    \lim_{\h\to 0} {\rm d}_H(A_\h,\,A_0) = 0.
    \]
  \end{itemize}
\end{definition}

\begin{remark}
  Let $(A_{\hbar})_{\hbar \in I}$ be a family of uniformly bounded
  subsets of $\mathbb{R}^n$. The limit set of
  $(A_{\hbar})_{\hbar \in I}$ is always a compact subset
  $A_0\subset\R^n$, and then $A_0$ is a Hausdorff limit of
  $(A_{\hbar})_{\hbar \in I}$.  Conversely, if a compact set $A_0$ is
  a Hausdorff limit of $(A_{\hbar})_{\hbar \in I}$, then it coincides
  with the limit set of $(A_{\hbar})_{\hbar \in I}$.
\end{remark}

\begin{figure}[h]
  \centering \label{process}
  \includegraphics[width=1\textwidth]{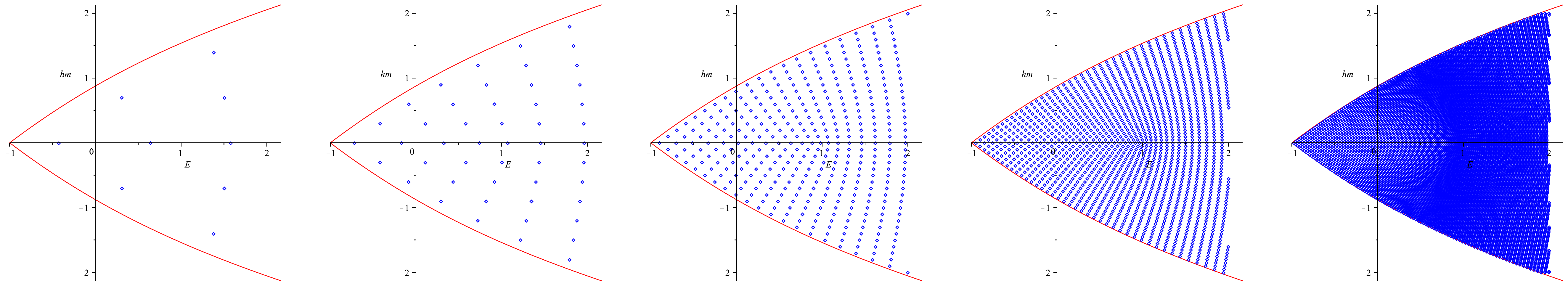}
  \caption{The figure depicts the semiclassical joint spectrum of the
    Quantum Spherical Pendulum for the following values of the Planck
    constant: $\hbar=0.7, 0.5, 0.3, 0.05, 0.02$. As $\hbar \to 0$, the
    semiclassical joint spectrum fills the inside of the red curve,
    which is the boundary of the classical spectrum of the system;
    this gives an illustration of the convergence stated in
    Theorem~\ref{theo:main}.  Notice that, in this figure, the joint
    spectrum is quite well behaved, because the operators form a
    completely integrable system. In this case, the joint spectrum is
    locally diffeomorphic to a lattice, as predicted by the
    Bohr-Sommerfeld rules, see~\cite{VN00}, and in fact much more than
    the classical spectrum can be recovered from the joint spectrum;
    see for instance~\cite{ChPeVN13, LFPeVN15}.}
  \label{newfigure}
\end{figure}

The following is the main theorem of this paper, which in some cases
strengthens previously known theorems (for instance in the case of
Berezin\--Toeplitz operators, and certain classes of pseuodifferential
operators).

\begin{theorem}
  \label{theo:main}
  Let $(T_1,\dots, T_d)$ be a family of pairwise commuting
  self\--adjoint semiclassical operators in the sense of
  Definition~\ref{sec:semicl-oper}. Then the limit set of the joint
  spectrum of $(T_1,\dots, T_d)$ is the classical spectrum of
  $(T_1,\dots, T_d)$.
\end{theorem}

\begin{proof}
  Let $\mu_j$ be the spectral measure of $T_j$, and let
  $\mu=\mu_1 \otimes \dots \otimes \mu_d$ be the joint spectral
  measure of $(T_1,\dots, T_d)$ on $\R^d$. For any
  $c=(c_1,\dots,c_d)\in\R^d$, let $\varphi_c:\R^d \to \R$ be defined
  by
  $$\varphi_c(x_1,\dots,x_d):=\sum_{i=1}^d (x_i-c_i)^2.$$ Then
\[
\varphi_c(T_1,\dots,T_d) = \int_{\R^d} \varphi_c(x)d\mu(x) =
\int_{\R}t d((\varphi_c)_*\mu)(t)
\]
The last equality implies that the spectrum of
$\varphi_c(T_1,\dots,T_d)$ is the support of $(\varphi_c)_*\mu$.

Now from Axioms {(Q\ref{item:one})} and {(Q\ref{item:square})}, the
principal symbol of $\varphi_c(T_1,\dots,T_d)$ is
$$\varphi_c(f_1,\dots,f_d)=\norm{F-c}^2$$ where $f_i$ is the principal
symbol of $T_i$. We see that $c\in\overline{F(M)}$ if and only if
$\inf \varphi_c(f_1,\dots,f_d) = 0$. By Lemma~\ref{lem:max},
\begin{equation}
  \inf_M \varphi_c(f_1,\dots,f_d) = \lim_{\h\to 0}\inf \textup{Spec}(\varphi_c(T_1,\dots,T_d)).
  \label{equ:inf}
\end{equation}

We have
$\textup{Spec}(\varphi_c(T_1,\dots,T_d)) =
\textup{supp}((\varphi_c)_*\mu)$.  Since $\varphi_c$ is continuous,
\begin{equation}
  \textup{supp}((\varphi_c)_*\mu) =
  \overline{\varphi_c(\textup{supp}(\mu))}.
  \label{equ:push-forward}
\end{equation}

Assume that $c$ is not in the limit set of the joint spectrum of
$(T_1,\dots,T_d)$. Thus there is a small ball around $c$ which is
disjoint from $\textup{JointSpec}(T_1,\dots,T_d)$ for $\h$ small
enough, which implies that there is some constant $\epsilon>0$ such
that
$$\inf(\varphi_c(\textup{JointSpec}(T_1,\dots,T_d)))>\epsilon.$$ Since
$\textup{JointSpec}(T_1,\dots,T_d) = \textup{supp}(\mu)$ , we get in
view of \eqref{equ:push-forward} that
$$\inf(\textup{supp}((\varphi_c)_*\mu)) \geq \epsilon.$$ Therefore, by
Equation~\eqref{equ:inf}, we get that
$$\inf_M \varphi_c(f_1,\dots,f_d) \geq \epsilon.$$ Hence
$c\not\in\overline{F(M)}$, which says that $\overline{F(M)}$ is
contained in the limit set of the joint spectrum.

In fact, all converse implications hold true, which proves the reverse
inclusion and hence the theorem.\hfill\rule{1ex}{1ex}
\end{proof}

\medskip

Theorem~\ref{theo:main} shows that the classical spectrum can be
recovered from the quantum joint spectrum. In the case of
Berezin-Toeplitz operators, this generalizes a result
of~\cite{PePoVN14}, where the convexity of the classical spectrum was
required.  For classes of pseudodifferential operators for which Axiom
{(Q\ref{item:square})} does not hold, we cannot apply
Theorem~\ref{theo:main}; however, the convex case holds, as we recall
below.

\begin{remark} \label{keyremark} We want to emphasize that axiom
  \emph{(Q\ref{item:square})}, while seemingly simple and quite close
  indeed to axiom \emph{(Q\ref{item:symbolic})}, gives in fact a great
  advantage in the form of a rudimentary (polynomial) functional
  calculus. We conjecture that the strong conclusion of
  Theorem~\ref{theo:main}, compared to
  Theorem~\ref{theo:pseudodifferential} below, could not be obtained
  by axioms \emph{(Q1--Q4)} alone.
\end{remark}

In what follows we work with $\h$-pseudodifferential operators which
are not necessarily bounded (see
Section~\ref{sec:pseudodifferential}).

\begin{theorem}[\cite{PePoVN14}]
  \label{theo:pseudodifferential}
  Let $X$ be either $\mathbb{R}^n$, or a closed manifold. Let
  $(T_1,\dots T_d)$ be a family of pairwise commuting self\--adjoint
  semiclassical $\hbar$\--pseudo\-diffe\-rential operators on $X$ whose
  symbols mildly depend\footnote{See
    Proposition~\ref{prop:major-examples}} on $\hbar$.  Let
  $\mathcal{S}\subset\R^d$ be the classical spectrum of
  $(T_1,\ldots,T_d)$. Suppose that $\mathcal{S}$ is a convex set.
  Then:
  \begin{itemize}
  \item from $ \textup{JointSpec} (T_1,\ldots,T_d) $ one can recover
    $\mathcal{S}$;
  \item if moreover each $T_i$, $1 \leq i \leq d$, is bounded, then
    $\overline{\mathcal{S}}$ is the Hausdorff limit, as $\h\to 0$, of
    $\textup{Convex Hull}(\textup{JointSpec}(T_{1,\hbar},\dots,
    T_{d,\hbar}))$.
  \end{itemize}
\end{theorem}

For further discussion on these results see \cite{PePoVN14,
  Pe2}. Notice that all the results presented in this paper strongly
rely on the self-adjointness of the operators, which ensures a stable
behaviour of the spectrum as $\hbar \to 0$. For general non-selfadoint
operators, for which there is considerable recent interest
(see~\cite{Tr,Sj1,Sj2}), similar results can probably be obtained for
the semiclassical pseudo-spectrum instead of the spectrum, but to the
authors knowledge, this has never been studied in the case of
commuting operators. On the other hand, for non-selfadjoint operators
that are normal, the stability of the spectrum is expected to hold,
see for instance~\cite{LFPe15b}.

\subsection{The completely integrable case}

Let $(M,\, \omega)$ be a $2n$\--dimensional symplectic manifold. Given
a smooth function $f \colon M \to \R$, we define the \emph{Hamiltonian
  vector field} $\mathcal{X}_f$ induced by $f$ on $M$ by
$$
\omega\,(\mathcal{X}_{f},\,\,\,\cdot)=-\,\,{\rm d} f.$$
This differential equation (or rather, system of differential
equations) is known as Hamilton's equation.

\begin{definition} \label{classical} A \emph{classical integrable
    system} on $(M,\, \omega)$ is given by a smooth
  $\mathbb{R}^n$\--valued map
$$
{F}:=(f_1,\, \ldots,\, f_n) \colon M \to \mathbb{R}^n
$$
such that each component function $f_i$ is constant along the flow of
the vector field $\mathcal{X}_{f_j}$ generated by the component $f_j$,
for all $i,j$, and the vector fields
$\mathcal{X}_{f_1}, \ldots, \mathcal{X}_{f_n}$ are linearly
independent almost everywhere.
\end{definition}

The first of the conditions in Definition~\ref{classical} can be
rephrased as $$\{f_i,f_j\}=0$$ for all $i,j$,
where $$\{f_i,f_j\}:=\omega(\mathcal{X}_{f_i},\mathcal{X}_{f_j})$$ are
the so called \emph{Poisson brackets} of $f_i$ and $f_j$; in this case
we say that $f_i$ and $f_j$ are in involution.  The integer $n$ (half
the dimension of $M$) in this definition is the largest integer for
which the conditions of the definition hold: that is, there is no set
of functions $$f_1,\ldots,f_k,$$ with $k > n$ which satisfies the
conditions above and it is in this sense that the word ``integrable
system" is used.

\begin{remark}
  The symplectic theory of finite dimensional integrable Hamiltonian
  systems relies on several fundamental
  results. Liouville\--Mineur\--Arnold's action\--angle
  theorem~\cite{mineur,arnold} is one of the fundamental pieces of the
  modern theory of integrable systems.  Duistermaat~\cite{Du1980}
  described the obstruction to the existence of global\--action
  coordinates in 1980, and this was the starting point of the global
  symplectic theory of integrable systems. Eliasson~\cite{El84,El90}
  proved in the 1980s a major theorem on the linearization of smooth
  non\--degenerate singularities of integrable systems, which
  continues to be one of the foundational and most useful results of
  the subjects; the majority of (but not all) results known to date
  about the general structure of integrable systems, assume that the
  singularities are non\--degenerate.
\end{remark}

In the 1980s the global classification of toric integrable systems of
Atiyah, Guillemin\--Sternberg, and Delzant opened up the doors and
served as inspiration to many authors working on global symplectic
invariants of integrable systems. Our next goal is to define toric
integrable systems, and the natural transformations between them, and
state the classification in the work of Atiyah, Guillemin\--Sternberg
and Delzant.  Let $(M,\omega)$ be a $2n$\--dimensional symplectic
manifold.  A smooth map
$$F:=(f_1,\dots,f_n) \colon M\to\mathbb{R}^n$$ on $(M,\omega)$ is a \emph{momentum map} for a
Hamiltonian $n$\--torus action if each of the Hamiltonian
flows $$t_j\mapsto \varphi_{f_j}^{t_j}$$ of the vector fields
$\mathcal{X}_{f_j}$ is periodic of period 1, and all of them pairwise
commute, that is,
$$\varphi_{f_j}^{t_j}\circ \varphi_{f_i}^{t_i} = \varphi_{f_i}^{t_i} \circ
\varphi_{f_j}^{t_j}$$
so that they define an action of the torus $\R^n/\Z^n$.

\begin{definition}
  We say that a \emph{momentum map} for a Hamiltonian $n$\--torus
  action is a \emph{toric integrable system}, or simply a \emph{toric
    system}, if in addition the following conditions hold:
  \begin{itemize}
  \item the manifold $M$ is closed and connected;
  \item the action of the torus $\R^n/\Z^n$ is effective.
  \end{itemize}
\end{definition}
 
The natural transformations between toric integrable systems preserve
the toric and the symplectic structure simultaneously, they are
precisely given by the following.
 
\begin{definition} \label{toriciso} Two toric systems $(M, \om ,F)$
  and $( M', \om' , F')$ are \emph{isomorphic} if there exists a
  symplectomorphism $\varphi: (M,\omega) \rightarrow (M',\omega')$
  such that $\varphi^* F' = F.$
\end{definition}

Atiyah and Guillemin\--Sternberg proved the following influential
result (in fact their result applied to much more general momentum
maps given by an $m$\--tuple on a $2n$\--manifold, where $n$ is not
necessarily equal to $m$ and the induced toral action is not
necessarily effective):

\begin{theorem}[\cite{At82,GuSt82}] \label{cos} The image $F(M)$ of a
  toric system $F \colon M \to \mathbb{R}^n$ is a convex polytope in
  $\mathbb{R}^n$.
\end{theorem}

The set of fixed point (also called elliptic points) of the induced
$\R^n/\Z^n$\--action is a collection of symplectic submanifolds of
$M$, and its image under $F$ gives a finite collection of points
$$p_1,\ldots,p_k \in \mathbb{R}^n \,\,\,\, k\geq 1.$$ The convex
polytope in Theorem~\ref{cos} is precisely the set
$$
\Delta=\textup{Convex Hull}(\{p_1,\ldots,p_k\}).
$$
Shortly after Atiyah and Guillemin\--Sternberg proved their theorem,
Delzant proved a converse type result, hence giving a classification
of toric systems.

\begin{theorem}[\cite{De88}] \label{del} The image $F(M)$ of a toric
  system $F \colon M \to \mathbb{R}^n$ is a Delzant polytope
  (i.e. rational, simple, and smooth). Moreover, $(M,\omega,F)$ is
  classified, up to isomorphisms, by $F(M)$.
\end{theorem}

Theorem~\ref{del} was generalized in \cite{PeVN09,PeVN11} to a class
of systems $f_1,f_2$ on four dimensional manifolds, called
\emph{semitoric systems}, in which only $f_1$ is required to generate
a periodic flow.

\begin{definition} \label{qi} A \emph{quantum integrable system} is
  given by a collection of $n$ commuting semiclassical self\--adjoint
  operators
  \[
  T_1:=(T_{1,\h})_{\h\in I},\,\ldots,\, T_n:=(T_{n,\h})_{\h \in I}
  \]
  whose principal symbols form a classical integrable system on $M$.
\end{definition}

\begin{definition}
  A quantum integrable system $T_1,\, \ldots,\, T_n$ on
  $(M,\, \omega)$ is \emph{toric} if the principal symbols of
  $T_1,\, \ldots,\, T_n$ are a toric system.
\end{definition}

\begin{remark}
  It is known that not every symplectic manifold has a complex
  structure or a prequantum line bundle. In the case of toric
  integrable systems, the situation is better. A toric integrable
  system does admit a compatible complex structure, which is however
  not unique. Suppose that $F:M\to\RM^n$ is the momentum map and
  let $$\Delta:=F(M)$$ be its image in $\mathbb{R}^n$, which is a
  convex polytope (by Theorem~\ref{cos}), say with
  vertices $$p_1,\ldots,p_k.$$ Then the system is is prequantizable if
  and only if there is a constant $\ell \in\RM^n$ such
  that $$p_1+\ell,\ldots,p_k+\ell \in 2 \pi \Z^n.$$ If this holds, the
  prequantum line bundle is in fact is unique, up to isomorphisms.
\end{remark}

In the case of Berezin-Toeplitz quantization, it is remarkable that
the joint spectrum of a toric system can be completely described, as
follows.  
\begin{theorem}[\cite{ChPeVN13}] \label{above}
  \label{theo:inverse-spectral}
  Let $T_1,\,\ldots,T_n$ be a Berezin-Toeplitz quantum toric system on
  a closed manifold $M$. Then
  $$
  \textup{JointSpec}(T_1,\dots, T_n) = g \Bigl( \Delta\cap \biggl( v +
  \frac{2 \pi}{k} \mathbb{Z}^n \biggr) ;\, k \Bigr) +
  \mathcal{O}(k^{-\infty})
  $$
  where:
  \begin{itemize}
  \item the set $\Delta \subset \mathbb{R}^n$ is the convex polytope
    $F(M)$ in the Atiyah\--Guillemin\--Sternberg theorem
    (Theorem~\ref{cos});
  \item the point $v \in \mathbb{R}^n$ is any vertex of $\Delta$;
  \item $g(\cdot;k):\R^n\to\R^n$ admits a
    ${\rm C}^{\infty}$\--asymptotic expansion
   $$
   g(\cdot;\,k) = \textup{Id}+k^{-1}g_1+k^{-2}g_2+\cdots
  $$
  where each $g_j:\R^n\to\R^n$ is a smooth function.
\end{itemize}
Moreover, if the spectral parameter $k$ is large enough then the
multiplicity of the eigenvalues of
  $$\textup{JointSpec}(T_1,\dots, T_n)$$ is precisely equal to $1$, and
  and there is $\delta>0$ such that if $\nu$ is an eigenvalue then the
  ball ${\rm B}(\nu, \delta/k)$ centered at $\nu$ or radius $\delta/k$
  contains precisely only the eigenvalue $\nu$.
\end{theorem}

As an immediate consequence of this theorem, we have the following.

\begin{corollary} \label{nice} Let $T_1,\,\ldots,T_n$ be a quantum
  toric system on a closed manifold $M$.  Then the joint spectrum of
  $T_1,\, \ldots,\, T_n$ modulo $\mathcal{O}(1/k)$ determines the
  classical integrable system given by principal symbols, up to
  isomorphisms.
\end{corollary}

A quicker alternative proof of Corollary~\ref{nice} was given in
\cite{PePoVN14}, and it also follows (in the same way as therein) from
Theorem~\ref{theo:main}. Indeed, Theorem~\ref{theo:main} implies that
from the joint spectrum one can recover $\overline{F(M)}$. But we know
that $\overline{F(M)}=\Delta$ is the Delzant polytope of the
system. Since, by Delzant's theorem~\ref{del}, the polytope is enough
to reconstruct the manifold and the moment map $F$, it follows that
the joint spectrum completely determines the classical
system.\hfill\rule{1ex}{1ex}

\medskip
  
To conclude, let us mention that an interesting consequence of the
proofs of these results is that any classical toric system can be
quantized.

\begin{corollary}[\cite{ChPeVN13}] \label{above} There exists a
  quantization of any classical toric integrable system. That is,
  given a classical toric system there is a quantum toric system whose
  principal symbols are precisely those given by the classical toric
  system.
\end{corollary}

We do not know whether a similar statement holds for more general
classes of completely integrable systems. In the analytic case, an
algebraic obstruction was constructed in~\cite{GaVa2010}.

\vspace{2mm}

\clearpage

\emph{Acknowledgements.}  We thank Y. Le Floch and J. Palmer for
comments on a preliminary version, and L. Polterovich for useful
discussions. Part of this paper was written while the second author
was holding the Lebesgue Chair at the Lebesgue Center, during the
Thematic Semester in Analysis and PDEs.

A.P. is partially supported by NSF CAREER grant DMS-1518420, Lebesgue
Chair 2015, and Severo Ochoa grant Sev-2011-0087.

V.N.S.  is partially supported by the Institut Universitaire de France
and the Lebesgue Center (ANR Labex LEBESGUE).

  \bigskip

  \noindent\'Alvaro Pelayo\\
  Department of Mathematics\\
  University of California, San Diego \\
  9500 Gilman Drive \#0112\\
  La Jolla, CA 92093-0112, USA\\
  alpelayo@math.ucsd.edu \\

  \medskip

  \noindent San V\~u Ng\d oc\\
  Institut de Recherches Math\'ematiques de Rennes\\
  Universit\'e de Rennes 1\\
  Campus de Beaulieu\\
  F-35042 Rennes cedex, France\\
  san.vu-ngoc@univ-rennes1.fr\\

\end{document}